\documentclass[11pt]{amsart}
\usepackage{amsmath}
\usepackage{amsthm}
\usepackage{geometry}           % See geometry.pdf to learn the layout options. There are lots.
\usepackage{graphicx}
\usepackage{amssymb}
\usepackage{epstopdf}
\newtheorem{lem}{Lemma}
\newtheorem{thm}{Theorem}

\newtheorem{prop}{Proposition}
\newtheorem{rmk}{Remark}

\title{Existence of Weak Conical K\"{a}hler-Einstein Metrics Along Smooth Hypersurfaces}
\author{Chengjian Yao}
\date{Aug 19 ,2013}                                           % Activate to display a given date or no date

\begin{document}
\maketitle

\begin{abstract}
The existence of \emph{weak conical K\"{a}hler-Einstein} metrics along smooth hypersurfaces with angle between $0$ and $2\pi$ is obtained by studying a smooth continuity method and a \emph{local Moser's iteration} technique. In the case of negative and zero Ricci curvature, the $C^0$ estimate is unobstructed; while in the case of positive Ricci curvature, the $C^0$ estimate obstructed by the properness of the \emph{twisted K-Energy}. As soon as the $C^0$ estimate is achieved, the local Moser iteration could improve the \emph{rough bound} on the approximations to a \emph{uniform $C^2$ bound} , thus produce a \emph{weak conical K\"{a}hler-Einstein} metric. The method used here do not depend on the bound of any background conical K\"{a}hler metrics.
\end{abstract}

\section{Introduction}
The existence of smooth K\"{a}hler-Einstein metrics with negative Ricci curvature is solved independently by Aubin and Yau, and Ricci-flat case by Yau in his celebrated work on the Calabi's conjecture.  In the last part of the paper of Yau \cite{Yau}, complex Monge-Amp$\grave{\text{e}}$re equation with degenerate and meromorphic right hand side was also considered. For the conical K\"{a}hler-Einstein metrics, the Riemann surfaces situation was well studied by Troyanov \cite{Troyanov} and McOwen \cite{McOwen}. The high dimensional was considered by Tian in \cite{Tian}. And in recent years, partly motivated by Donaldson's project of using conical K\"{a}her-Einstein metrics as a continuity method of solving smooth K\"{a}hler-Einstein problem on Fano manifolds, lots of results along this line have been achieved, see \cite{Campana}, \cite{Henri}, and \cite{Yanir}.  Quite recently, Chen, Donaldson and Sun \cite{Chen1},\cite{Chen2},\cite{Chen3} proved the smooth Yau-Tian-Donaldson conjecture in Fano case, where conical K\"{a}hler-Einstein metrics are used as a main ingredient. The existence of conical K\"{a}hler-Einstein metrics still has its own interest in studying more general compact K\"{a}hler manifolds.\\

The aim of this paper is to study the existence of \emph{weak conical K\"{a}hler-Einstein} metrics with conical singularities along smooth hypersurfaces.  By saying a \emph{weak conical K\"{a}hler-Einstein} metric $\omega$ on $X$ with cone angle $2\pi \beta$ along a smooth hypersurface $D$, we mean that $\omega$ is smooth K\"{a}hler-Einstein on $X\backslash D$, and near each point $p\in D$, there exists coordinate neighborhood $(z,z_2,\cdots, z_n)\in U$ such that locally $D$ is defined by $\{z=0\}$, and $\omega$ is quasi-isometric to a local model metric $\omega_\beta=\sqrt{-1}\{|z|^{2\beta-2}\mathrm{d}z\wedge\mathrm{d}\bar{z}+\sum_{i=2}^n \mathrm{d}z_i\wedge\mathrm{d}\bar{z}_i\}$ on $U$.

By studying the smoothing continuity path and proving uniform $C^2$ estimate in terms of local model conical metric along the hypersurface, we could get \emph{weak conical K\"{a}hler-Einstein} metrics. Let $X$ be a compact K\"{a}hler manifold, $\omega_0$ be a smooth K\"{a}hler metric on $X$, and $D$ be a smooth hypersurface, $\beta\in (0,1)$, suppose the cohomological condition $c_1(X)=\mu [\omega_0]+(1-\beta)c_1(L_D)$ is satisfied. By studying a smoothing of the current equation $Ric(\omega_\phi)=\mu\omega_\phi+(1-\beta)\chi$, we get the main theorem of this paper

\begin{thm}There are three situations:\\
1.For $\mu<0$, there exists \emph{weak conical K\"{a}hler-Einstein} metric with Ricci curvature $\mu$;\\
2.For $\mu=0$, there exists \emph{weak conical K\"{a}hler-Einstein} metric with zero Ricci curvature;\\
3.For $\mu>0$, under the assumption that $c_1(L_D)$ is nonnegative and the \emph{twisted K-Energy} is proper, there exists \emph{weak conical K\"{a}hler-Einstein} metric with Ricci curvature $\mu$.\\
\end{thm}

Brendle proved the existence of Ricc-flat K\"{a}hler metrics with conical singularity along a smooth hypersurface under the assumption that the cone angle $\beta <\frac{1}{2}$. A group of authors Campana, Guenancia and P$\breve{\text{a}}$un \cite{Campana}  dealt with the general situation where the divisor $D$ is allowed to have simple normal crossing, under the assumption that $\beta \leq \frac{1}{2}$. Around the same time, another group of authors Jeffres, Mazzeo and Rubinstein obtained the existence of \emph{conical K\"{a}hler-Einstein} metrics with conical singularities along a smooth hypersurface  by studying the classical continuity path of conical singular metrics for general angle $\beta \in (0,1)$. \\

The method adopted here differs from the work of Brendle \cite{Brendle} and Jeffres, Mazzeo and Rubinstein\cite{Yanir} in the sense that by smoothing the \emph{delta distribution} $\chi$ in the equation, we avoided dealing with continuity path of singular metrics. \\

In Section 2 we primarily set up the problem of solving conical K\"{a}hler-Einstein metric with positive Ricci curvature on a Fano manifold, and we also outline how uniform $C^0$ estimate is obtained for the \emph{Smooth Continuity Path}. Then following the idea of using Chern-Lu inequality as in Chen, Donaldson and Sun \cite{Chen1}, a rough $C^2$ bound is derived. Section 3 explains how to go from this rough $C^2$ bound to a finer $C^2$ bound $C^{-1}\omega_{\beta,\epsilon}\leq \omega_{\phi_\epsilon} \leq C\omega_{\beta,\epsilon}$ by using a local Moser iteration.  In section 4,  it is illustrated how this idea could be applied to the general case of the current equation in the main theorem. The last section discusses some questions about the regularity problem of the \emph{weak conical K\"{a}hler-Einstein} metric. \\

I was motivated to use a smooth Continuity method after reading the series of papers \cite{Chen1},\cite{Chen2},\cite{Chen3} by Chen, Donaldson and Sun and the hope that the angle assumption $\beta\leq \frac{1}{2}$ in Campana, Guenancia and P$\breve{\text{a}}$un \cite{Henri} could be removed by doing Yau's estimate locally. While preparing this note, I noticed the beautiful paper of  Guenancia and P$\breve{\text{a}}$un where they removed the assumption on $\beta$ and proved the existence of \emph{weak conical K\"{a}hler-Einstein} metrics for the more general case of simple normal crossing divisors.  However, the method used here, the \emph{local Moser iteration}, which could deal with the situation of smooth hypersurfaces, is different from their work and hopefully this is still of some interest.\\

$\mathbf{Acknowledgement}$: I would like to express great thanks to Professor Xiuxiong Chen, who gave me constantly encouragement and showed great patience during working on this problem. I would also like to thank Dr. Song Sun, Dr. Haozhao Li  and Dr. Yuanqi Wang for helpful discussion and providing useful notes at some points. The discussion with Yu Zeng and Long Li also helped me clarify several things which were not clear to me before.

\section{Setting Up for Fano Case}
Suppose $X$ is a smooth Fano manifold and $D$ is a smooth anti-canonical divisor, following \cite{Henri}, we define a \emph{weak conical K\"{a}hler-Einstein metric} with conical angle $2\pi \beta$ to be a metric which is smooth K\"{a}hler-Einstein on $X\backslash D$ and quasi-isometric to a local model metric $\sqrt{-1}\{\beta^2|z|^{2\beta-2}\mathrm{d}z\wedge\mathrm{d}\bar{z}+\sum_{i=2}^n\mathrm{d}z_i\wedge\mathrm{d}\bar{z}_i\}$ near the divisor $D$ where $z=0$ is local defining function of $D$.

Under the assumption the \emph{Twisted K-Energy} (see \cite{Chen1} or \cite{Li}) is proper,the existence of \emph{weak
K\"{a}hler-Einstein} metric with conical angle $2\pi \beta$ for $\beta\in (0,1)$ is established by a slightly different way from \cite{Yanir} in the sense that we are solving \emph{smooth} continuity method with uniform $C^2$ bound instead of dealing with the continuity path for the conical singular metrics directly.

Take $\omega_0$ to be a fixed smooth background K\"{a}hler metric in $c_1(X)$ and take a Hermitian metric $h$ on $L_D$ with curvature form $\omega_0$,  then we could consider the Continuity Path:

%%%
%%%Set up smooth continuity method first
%%%

\begin{equation}
\text{Ric } \omega_{\phi_\epsilon(t)}=t \omega_{\phi_\epsilon(t)}+(\beta-t)\omega_0+(1-\beta)\chi_\epsilon
\end{equation}

where the starting point is $\omega_{\phi_\epsilon(0)}=\omega_{\varphi_\epsilon}$ satisfying $\text{Ric } \omega_{\varphi_\epsilon}=\beta \omega_0+(1-\beta)\chi_\epsilon$, and $\chi_\epsilon=\omega_0+\sqrt{-1}\partial\bar{\partial}\text{log}(|S|_h^2+\epsilon)$ is smooth positive closed $(1,1)$ form in $c_1(X)$ approaching the $\emph{Delta Distribution $\chi$}$ along $D$. This could be done following the Calabi's \emph{Volume form Conjecture} solved by Yau.

Solve the smooth K\"{a}hler-Einstein problem is well-known to be obstructed by some \emph{Stability Condition} (See the recent work relating K-Stability and K\"{a}hler-Einstein metrics on Fano case by \cite{Chen1}, \cite{Chen2} and \cite{Chen3}) algebraically and \emph{Properness} of K-Energy \cite{Tian} analytically). 

Similarly, in the case of conical K\"{a}hler-Einstein situation, analytically we also need the properness of some functional, namely the \emph{Twisted K-Energy}. Let's recall the definitions:

%%%%
%%%  Definition of the functionals

$$E(\phi)=-n\int_0^1\mathrm{d}t\int_X \dot{\phi}_t (Ric(\omega_{\phi_t})-\omega_{\phi_t})\wedge\omega_{\phi_t}^{n-1}$$

$$E_{(1-\beta)D}(\phi)=E(\phi)+(1-\beta)J_\chi(\phi)$$

$$E_{\epsilon,(1-\beta)D}(\phi)=E(\phi)+(1-\beta)J_{\chi_\epsilon}(\phi)$$

$$J_\chi(\phi)=n\int_0^1\mathrm{d}t\int_X \dot{\phi}_t (\chi-\omega_{\phi_t})\wedge\omega_{\phi_t}^{n-1}$$

$$J_{\chi_\epsilon}(\phi)=n\int_0^1\mathrm{d}t\int_X \dot{\phi}_t (\chi_\epsilon-\omega_{\phi_t})\wedge\omega_{\phi_t}^{n-1}$$

 for any path of smooth potential $\phi_t$ connecting $\omega_0$ and $\omega_\phi$. And by saying \emph{Proper}, we mean that there exists constant $C_1$ and $C_2$ s.t. $E_{(1-\beta)D}(\phi)\geq C_1 (I-J)(\phi)-C_2$, where $I, J$ are the classical functionals used in the reference, for instance \cite{Li}.

\begin{prop}
If $E_{(1-\beta)D}$ is proper, then the above continuity path is solvable up to $t=\beta$ for any $\epsilon \in (0,1]$.
\end{prop}

\begin{proof}
We assume $\omega_0$ has volume 1. Along the continuity path, the Twisted K-Energy is decreasing by a direct computation as soon as $t<\beta$. Since the initial metric satisfies the Monge-Amp$\grave{\text{e}}$re equation $\omega_{\varphi_\epsilon}^n=e^{h_{\omega_0}+C_\epsilon}\frac{\omega_0^n}{(|S|_h^2+\epsilon)^{1-\beta}}$, where $C_\epsilon$ is chosen constant (bounded above and below) so that $\int_X e^{h_{\omega_0}+C_\epsilon}\frac{\omega_0^n}{(|S|_h^2+\epsilon)^{1-\beta}}=\int_X \omega_0^n=1$,   for  $1<p<\frac{1}{1-\beta}$,
\begin{align*}
\int_X (\frac{e^{h_{\omega_0}+C_\epsilon}}{(|S|_h^2+\epsilon)^{1-\beta}})^p\omega_0^n
&=\int_X(\frac{e^{h_{\omega_0}}}{(|S|_h^2+\epsilon)^{1-\beta}})^p \omega_0^n\slash (\int_X \frac{e^{h_{\omega_0}}}{(|S|_h^2+\epsilon)^{1-\beta}} \omega_0^n)^p\\
&\leq C\int_X(\frac{e^{h_{\omega_0}}}{(|S|_h^2+\epsilon)^{1-\beta}})^p \omega_0^n\\
&\leq C
\end{align*}

Kolodziej's theorem tells us that $\varphi_\epsilon$ has uniform $C^\alpha$ bound since the right hand side of the equation is uniform in $L^p$ for some $p>1$. 

By an explicit formula(see \cite{Li}) of the \emph{Twisted K-Energy}: 

$$E_{\epsilon,(1-\beta)D}(\phi)=\int_X \emph{log }\frac{\omega_\phi^n}{\omega_0^n} \omega_\phi^n -\beta(I-J)(\phi)+\int_X \{h_{\omega_0}-(1-\beta)\text{log }(|S|_h^2+\epsilon)\}(\omega_0^n-\omega_\phi^n).$$

For the initial K\"{a}hler metrics $\omega_{\varphi_\epsilon}$, 

\begin{align*}
E_{\epsilon,(1-\beta)D}(\varphi_\epsilon)
&=\int_X \emph{log }\frac{\omega_{\varphi_\epsilon}^n}{\omega_0^n} \omega_{\varphi_\epsilon}^n -\beta(I-J)(\varphi_\epsilon)+\int_X \{h_{\omega_0}-(1-\beta)\text{log }(|S|_h^2+\epsilon)\}(\omega_0^n-\omega_{\varphi_\epsilon}^n)\\
&=\int_X C_\epsilon \omega_{\phi_\epsilon}^n+\int_X\{h_{\omega_0}-(1-\beta)\text{log }(|S|_h^2+\epsilon)\} \omega_0^n-\beta (I-J)(\varphi_\epsilon)\\
&\leq C
\end{align*}

where the last inequality holds since the other classical functionals $I$, $J$ are bounded as soon the potential has $C^0$ bound. By the assumption that $E_{(1-\beta)D}$ is proper, we know that $E_{\epsilon,(1-\beta)D}(\phi)\geq E_{(1-\beta)D}(\phi)-C\geq C_1 (I-J)(\phi)-C_2$, then similarly as the proof on Page 12 in \cite{Chen1}, we could solve the continuity path up to $t=\beta$ with uniform $C^0$ bound on $\phi_\epsilon(\beta)=\phi_\epsilon$. 

\end{proof}

%%%%
%%%%
%%%    We finished the proof of C^0 bound
%%

The solution at the endpoint of the continuity method $\omega_{\phi_\epsilon(\beta)}=\omega_{\phi_\epsilon}$ satisfy Monge-Amp$\grave{\text{e}}$re equation for suitable normalization on $\phi_\epsilon$:

\begin{equation}
\omega_{\phi_\epsilon}=e^{-\beta \phi_\epsilon+h_{\omega_0}}\frac{\omega_0^n}{(|S|_h^2+\epsilon)^{1-\beta}}
\end{equation}

%%%
%%%Apriori Estimate Independent of \epsilon  
%%%
%%%   C^0 estimate give us rough bound
%%%

With $C^0$ bound in hand, by using global Chern-Lu inequality, a \emph{rough bound} could be derived, see \cite{Chen1} for example. For the reader's convenience, we outline the proof here. Keep in mind that from now on, all the constant $C$'s are constant varying from line to line, which are independent of $\epsilon$.

\begin{lem}{(Chern-Lu Inequality)}
Suppose $\omega$ and $\eta$ are two K\"{a}hler metrics on a compact K\"{a}hler manifold, if $Ric(\omega)\geq C_1\omega-C_2 \eta$ and the holomorphic bisectional curvature $R^\eta_{i\bar{j}k\bar{l}}\leq C_3(h_{i\bar{j}}h_{k\bar{l}}+h_{i\bar{l}}h_{k\bar{j}})$, then $$\Delta_\omega \text{log } \text{tr}_\omega \eta \geq C_1-(C_2+2C_3)\text{tr}_\omega \eta.$$
\end{lem}

\begin{prop}{(\cite{Chen1})}
$\exists C$, s.t. $C^{-1}\omega_0\leq \omega_{\phi_\epsilon}\leq \frac{C}{(|S|_h^2+\epsilon)^{1-\beta}}\omega_0$.
\end{prop}

\begin{proof}
Since  on $X\backslash D$,
\begin{align*}
\chi_\epsilon=\omega_0+\sqrt{-1}\partial \bar\partial \text{log }(|S|_h^2+\epsilon)
&=\omega_0+\sqrt{-1}\partial\frac{|S|_h^2\bar{\partial}\text{log }|S|_h^2}{|S|_h^2+\epsilon}\\
&=\omega_0+\frac{|S|_h^2}{|S|_h^2+\epsilon}\sqrt{-1}\partial\bar\partial \text{log }|S|_h^2+\sqrt{-1}\frac{\epsilon}{|S|_h^2(|S|_h^2+\epsilon)^2}\partial |S|_h^2\wedge \bar\partial |S|_h^2\\
&\geq \omega_0+\frac{|S|_h^2}{|S|_h^2+\epsilon}(-\omega_0)\\
&\geq 0
\end{align*}
 the smooth $(1,1)$ form $\chi_\epsilon\geq 0$

Because $Ric(\omega_{\phi_\epsilon})=\beta \omega_{\phi_\epsilon}+(1-\beta)\chi_\epsilon\geq \beta \omega_{\phi_\epsilon}$ and the bisectional curvature of $\omega_0$ is bounded above by $\Lambda$, the Chern-Lu inequality in the above lemma tells

$$\Delta_{\omega_{\phi_\epsilon}} \text{log } \text{tr}_{\omega_{\phi_\epsilon}}\omega_0 \geq \beta-2\Lambda \text{tr}_{\omega_{\phi_\epsilon}}\omega_0.$$ 

And by using the fact that $\Delta_{\omega_{\phi_\epsilon}}\phi=\text{tr}_{\omega_{\phi_\epsilon}} (\omega_{\phi_\epsilon}-\omega_0)=n-\text{tr}_{\omega_{\phi_\epsilon}}\omega_0$, we get the inequality

$$\Delta_{\omega_{\phi_\epsilon}} \{ \text{log } \text{tr}_{\omega_{\phi_\epsilon}} \omega_0 -(2\Lambda+1)\phi_\epsilon \} \geq \text{tr}_{\omega_{\phi_\epsilon}}\omega_0+\beta-n(2\Lambda+1).$$ The maximum principle on this compact manifold tells us $\text{tr}_{\omega_{\phi_\epsilon}}\omega_0 \leq \{n(2\Lambda+1)-\beta\} e^{(2\Lambda+1)\text{Osc }\phi_\epsilon}$, thus $\omega_0\leq C \omega_{\phi_\epsilon}$ and the bound in the other direction follows from this bound together with the Monge-Amp$\grave{\text{e}}$re equation.
\end{proof}

The upper bound in this control is satisfactory in the perpendicular direction while it is still far away from the correct control in the tangential direction. In the next section, we are going to apply local Moser iteration to improve the upper bound of $\omega_{\phi_\epsilon}$.

\section{Uniform Laplacian Estimate}

\subsection{A local smooth approximate metric}

Since we are willing to get a comparison between the metric $\omega_{\phi_\epsilon}$ and the 
local model conic metric $\omega_{\beta,\epsilon} =\sqrt{-1}\big\{(|z|^2+\epsilon)^{\beta-1}{\mathrm{d}z\wedge \mathrm{d}\bar{z}}+\sum_{i=2}^n\mathrm{d}z_i\wedge \mathrm{d}\bar{z}_i\big\}$ only in a local sense, to try to apply some kind of local maximum principle is not so easy. Instead of getting the upper bound from \emph{maximum principle},  we apply local Moser iteration to overcome this difficulty. 

We want to compare the K\"{a}hler metric $\omega_{\phi_\epsilon}$ with the regularized metric $\omega_{\beta,\epsilon}$ above locally on a coordinate balls near the divisor. 
The \emph{rough bound} $C^{-1} \omega_{Euc} \leq \omega_{\phi_\epsilon} \leq \frac{C}{(|S|_h^2+\epsilon)^{1-\beta}} \omega_{Euc}$ implies a \emph{rough geometry}.

\begin{lem} There is H\"{o}lder control for the distance function $\omega_{\phi_\epsilon}$:\\
1. For any $R>0$, $B_{\phi_\epsilon}(p, R) \subset B_{Euc}(p,\sqrt{C}R)$;\\
2. For any $R>0$, $\exists \delta=\delta(R)(=\tilde{C}R^\frac{1}{\beta})>0, \epsilon(R)>0$, s.t. $\forall \epsilon <\epsilon(R)$, $B_{Euc}(p, \delta) \subset B_{\phi_\epsilon}(p,R)$, where $\tilde{C}=\{(\frac{2}{\beta}+2)\sqrt{C}\}^{-\frac{1}{\beta}}$.
\end{lem}

\begin{proof}
The first part is just a direct consequence of the bound $\omega_{Euc} \leq C \omega_{\phi_\epsilon}$, any curve initiating from $p$ measuring under the metric $\omega_{\phi_\epsilon}$ with length less than $R$ will have length less than $\sqrt{C}R$ when measured using the Euclidean metric.

The second part of the containing relationship is also not so difficult. On an Euclidean ball $B_{Euc}(p,\delta)=\{(z,z_2,\cdots,z_n)||z|\leq \delta, |z_i|\leq \delta\}$ centered at $p$, pick any point $q$, then $q$ is joined with $p$ by detouring line segments $\overline{qq'}, \overline{q'p'}, \overline{p'p}$, where $q'$ is the point on the boundary of the Euclidean ball with the same $z_2,\cdots,z_n$ coordinates as $q$ and $z$ coordinate just the radial projection of the $z$ coordinate of $q$ to the boundary, and $p'$ is just the projection of $q'$ to the complex line $z_2=0, \cdots, z_n=0$. When measured under the metric $\omega_{\phi_\epsilon}$, $|\overline{qq'}|\leq \int_{r=0}^\delta \sqrt{C}(r^2+\epsilon)^{\frac{\beta-1}{2}} \mathrm{d}r \leq \sqrt{C} \frac{1}{\beta} \delta^\beta$ and similarly $|\overline{pp'}|\leq \sqrt{C} \frac{1}{\beta} \delta^\beta$ and $|\overline{q'p'}|\leq \int_{z'\in \overline{q'p'}} \sqrt{C}(\delta^2+\epsilon)^\frac{\beta-1}{2}|\mathrm{d}z'|\leq \sqrt{C} (\delta^2+\epsilon)^\frac{\beta-1}{2}\delta \leq \sqrt{C} (\delta^2+\epsilon)^\frac{\beta}{2}$. Therefore, if we choose $\delta$ so that $(\frac{2}{\beta}+2)\sqrt{C}\delta^\beta<R$ and choose $\epsilon(R)=\delta^2$, then the distance of $q$ and $p$ under the metric $\omega_{\phi_\epsilon}$ will be no bigger than $R$. 
\end{proof}

Also, as is proved in \cite{Chen1}, the \emph{rough bound} $C^{-1}\omega_0\leq \omega_{\phi_\epsilon}\leq \frac{C}{(|S|_h^2+\epsilon)^{1-\beta}}\omega_0$ gives the bound on diameter.

\begin{lem}
The diameter of $(X,\omega_{\phi_\epsilon})$ is bounded above.
\end{lem}
%%%%
%%%%  Geometry extracted from the rough estimate
%%%%
%%%%

From this  \emph{rough geometry}, a finite collection of covering on $D$ which Moser Iteration will be applied could be chosen. Firstly, around each point $p_\alpha\in D$, there exists coordinate ball $\mathcal{U}_\alpha(p_\alpha)=\tau_\alpha(B_\alpha(0,r_\alpha))$ where $B_\alpha(0,r_\alpha)=\{(z^\alpha=z^\alpha_1,z^\alpha_2,\cdots,z^\alpha_n)||z^\alpha|, |z^\alpha_2|, \cdots, |z^\alpha_n|< r_\alpha\}\subset \mathbb{C}^n$, such that $\mathcal{U}_\alpha \cap D$ corresponds to $\{z^\alpha=0\}$, the first item of Lemma 2 tells that 
$\mathcal{U}_\alpha$ contains $\omega_{\phi_\epsilon}$-metric ball $U_\alpha=B_{\phi_\epsilon}(p_\alpha,\frac{r_\alpha}{\sqrt{C}})$, and the second item of Lemma 2 tells this metric ball contains Euclidean ball $B_{Euc}(p_\alpha,\delta(\frac{r_\alpha}{\sqrt{C}}))$ and again this contains $\omega_{\phi_\epsilon}$-metric ball $V_\alpha=B_{\phi_\epsilon}(p_\alpha,\frac{1}{\sqrt{C}}\delta(\frac{r_\alpha}{\sqrt{C}}))$, $V_\alpha$ contains coordinate ball $\mathcal{V}_\alpha(p_\alpha)=B_{Euc}(p_\alpha, s_\alpha)$, for $s_\alpha=\delta(\frac{1}{\sqrt{C}}\delta(\frac{r_\alpha}{\sqrt{C}}))$. Since $D$ is compact, we could cover it by finitely many $\mathcal{V}_\alpha(p_\alpha)$'s, and the weight for the Hermitian metric $e^{-h_\alpha}$ satisfies $|h_\alpha|\leq C_\alpha, |\frac{\partial h_\alpha}{\partial z_i^\alpha}|\leq C_\alpha, |\frac{\partial^2 h_\alpha}{\partial z_i^\alpha\partial \bar{z}_i^\alpha}|\leq C_\alpha$.

Then on each $\mathcal{U}_\alpha$, take the smooth K\"ahler metrics $\omega_{\beta,\epsilon}^\alpha=\sqrt{-1}\{\beta^2(|z^\alpha|^2+\epsilon)^{\beta-1}\mathrm{d}z^\alpha\wedge\mathrm{d}\bar{z}^\alpha+\sum_{j=2}^n \mathrm{d}z^\alpha_j\wedge\mathrm{d}\bar{z}^\alpha_j\}$ as the approximating metrics to the standard flat conical metric. From now on, for simplicity the notations will forget about $\alpha$. For those local metrics,  $R_{1\bar{1}1\bar{1}}=\epsilon (1-\beta)(|z|^2+\epsilon)^{\beta-3}$, and $ R_{i\bar{j}k\bar{l}}=0$ if one of $i,j,k,l \neq 1$, therefore the holomorphic bisectional curvature $R(\xi,\bar{\xi},\eta,\bar{\eta})=|\xi^1|^2|\eta^1|^2R_{1\bar{1}1\bar{1}}\geq 0$ for any pair of $(1,0)$-vector fields $\xi=\xi^1\frac{\partial}{\partial z}+\sum_i\xi^i\frac{\partial}{\partial z^i}$, and $\eta=\eta^1\frac{\partial}{\partial z}+\sum_i\eta^i\frac{\partial}{\partial z^i}$.  %%
%%Laplacian estimate for the Kahler metrics!
%%

%%%%
%%%%
%%%%Get the differential Inequality
%%%%

\subsection{Differential Inequality}

The goal of this section is to derive a local differential inequality on $\tilde{U}_\alpha$. Let $\sigma_{\epsilon,\beta}=\text{tr}_{\omega_{\beta,\epsilon}} \omega_{\phi_\epsilon}$, then

\begin{prop}
$$\Delta_{\omega_{\phi_\epsilon}} \emph{log } \sigma_{\epsilon,\beta}\geq -C(|z|^2+\epsilon)^{-\beta}-C$$
\end{prop}

To prove this inequality we need a lemma.

\begin{lem}
$$
\Delta_{\omega_{\phi_\epsilon}} \sigma_{\epsilon,\beta} \geq -C \sigma_{\epsilon,\beta}-C+\sum_{i,j,k}\frac{1}{1+\phi_{i\bar{i}}}\frac{1}{1+\phi_{j\bar{j}}}|\phi_{i\bar{j}k}|^2-C(|z|^2+\epsilon)^{-\beta}
$$
\end{lem}

%%%%
%%%    Proof of the lemma
%%%

\begin{proof}{(Proof of Lemma 2)}
First, on $\mathcal{U}_\alpha$, suppose $\omega_{\beta,\epsilon}=\omega_0+\sqrt{-1}\partial\bar{\partial}\psi_\epsilon$ with suitable normalization on the potential $\psi_\epsilon$ such that $||\psi_\epsilon||_{C^0}\leq C$, rewrite the equation $(2)$, 

\begin{equation}
(\omega_{\beta,\epsilon}+\sqrt{-1}\partial \bar \partial\tilde{\phi}_\epsilon)^n=e^{F_{\epsilon,\beta}}\omega_{\beta,\epsilon}^n
\end{equation}

where $F_{\epsilon,\beta}=-\beta \phi_\epsilon+h_{\omega_0}+(1-\beta)h+(1-\beta)\text{log}\frac{|z|^2+\epsilon}{|z|^2+\epsilon e^h}+\text{log }\frac{\omega_0^n}{\omega_{Euc}^n}.$

Following Yau's calculation on the $C^2$ estimate, use $g_{i\bar{j}}$ to denote the background metric $\omega_{\beta,\epsilon}$ and $g'_{i\bar{j}}$ to denote $\omega_{\phi_\epsilon}$, then under the local normal coordinate for which $g_{i\bar{j}}=\delta_{i\bar{j}}, g_{i\bar{j}}'=(1+\phi_{i\bar{j}})\delta_{i\bar{j}}$, 

\begin{align*}
\Delta_{\omega_{\phi_\epsilon}} \sigma_{\epsilon,\beta}
&=\Delta_{\omega_{\beta,\epsilon}} F_{\epsilon,\beta} +\sum_{i,j,k,l,n}g^{'k\bar{j}}g^{'i\bar{n}}\phi_{k\bar{n}\bar{l}}\phi_{i\bar{j}l}+\sum_{k,l}R_{k\bar{k}l\bar{l}}\frac{(\phi_{k\bar{k}}-\phi_{l\bar{l}})^2}{(1+\phi_{k\bar{k}})(1+\phi_{l\bar{l}})}\\
&\geq \Delta_{\omega_{\beta,\epsilon}} F_{\epsilon,\beta}+\sum_{i,j,k}\frac{1}{1+\phi_{i\bar{i}}}\frac{1}{1+\phi_{j\bar{j}}}|\phi_{i\bar{j}k}|^2
\end{align*}

Since $h_{\omega_0}$, $h$ and $\text{log }\frac{\omega_0^n}{\omega_{Euc}^n}$ are smooth functions, and $\Delta_{\omega_{\beta,\epsilon}}\phi_\epsilon=\text{tr }_{\omega_{\beta,\epsilon}} (\omega_{\phi_\epsilon}-\omega_0)\leq \sigma_{\epsilon,\beta}$, our primarily care should be the lower bound on $\Delta_{\omega_{\beta,\epsilon}} \text{log} \frac{(|z|^2+\epsilon)}{(|z|^2+\epsilon e^h)}$.

\begin{align*}
\sqrt{-1}\partial\bar\partial \text{log} (|z|^2+\epsilon)=&\sqrt{-1}\frac{\epsilon \mathrm{d}z\wedge \mathrm{d}\bar{z}}{(|z|^2+\epsilon)^2}\\
\sqrt{-1} \partial\bar\partial \text{log}(|z|^2+\epsilon e^h)=&\sqrt{-1}\{\frac{\epsilon e^h+\epsilon e^h|z|^2\frac{\partial^2 h}{\partial z\partial \bar{z}}+\epsilon^2 e^{2h}\frac{\partial^2 h}{\partial z\partial \bar{z}}+\epsilon e^h|z|^2|\frac{\partial h}{\partial z}|^2-\epsilon e^h(\bar{z}\frac{\partial h}{\partial \bar{z}}+z\frac{\partial h}{\partial z})}{(|z|^2+\epsilon e^h)^2}\mathrm{d}z\wedge\mathrm{d}\bar{z}\\
&+\sum_{i=2}^n\frac{\epsilon e^h |z|^2\frac{\partial^2 h}{\partial \bar{z}_i\partial z_i}+\epsilon^2 e^{2h}\frac{\partial ^2 h}{\partial \bar{z}_i\partial z_i}+\epsilon e^h |z|^2 \frac{\partial h}{\partial \bar{z}_i}\frac{\partial h}{\partial z_i}}{(|z|^2+\epsilon e^h)^2} \mathrm{d}z_i\wedge\mathrm{d}\bar{z}_i\}\\
&+\text{mixed terms}\\
&\leq C\sqrt{-1} \{\frac{\epsilon}{(|z|^2+\epsilon)^2} \mathrm{d}z\wedge\mathrm{d}\bar{z}+\sum_{i=2}^n \mathrm{d}z_i\wedge\mathrm{d}\bar{z}_i\}+\text{mixed terms}\\
&\leq C\sqrt{-1}\{\frac{1}{|z|^2+\epsilon}\mathrm{d}z\wedge\mathrm{d}\bar{z}+\sum_{i=2}^n\mathrm{d}z_i\wedge\mathrm{d}\bar{z}_i\}+\text{mixed terms}
\end{align*}

Because $\omega_{\beta,\epsilon}$ is diagonal metric, the mixed terms do not contribute when take the trace, thus we get 
$\Delta_{\omega_{\beta,\epsilon}}\text{log}\frac{|z|^2+\epsilon}{|z|^2+\epsilon e^h}\geq -C\{(|z|^2+\epsilon)^{1-\beta}(|z|^2+\epsilon)^{-1}+(|z|^2+\epsilon)^{1-\beta}\}\geq -C(|z|^2+\epsilon)^{-\beta}-C$, and the Lemma follows. 

\end{proof}

%%%%%
%%%%
%%%   PROOF of the Proposition
%%
%

\begin{proof}{(Proof of Proposition 3)}
\begin{align*}
|\nabla \sigma_{\epsilon,\beta}|_{\omega_{\phi_\epsilon}}^2
&=\sum_{i,j,k,l,p,q} g^{'i\bar{j}}g^{k\bar{l}}\phi_{k\bar{l}i}g^{p\bar{q}}\phi_{p\bar{q}\bar{j}}\\
&=\sum_i \frac{1}{1+\phi_{i\bar{i}}}|\sum_k \phi_{k\bar{k}i}|^2\\
&=\sum_i \frac{1}{1+\phi_{i\bar{i}}}|\sum_{k}\frac{\phi_{k\bar{k}i}}{(1+\phi_{k\bar{k}})^{1/2}}(1+\phi_{k\bar{k}})^{1/2}|^2\\
&\leq \sum_i \frac{1}{1+\phi_{i\bar{i}}}\sum_k \frac{|\phi_{k\bar{k}i}|^2}{1+\phi_{k\bar{k}}}\sum_l (1+\phi_{l\bar{l}})\\
&=\sigma_{\epsilon,\beta} \sum_{i,k}\frac{1}{1+\phi_{i\bar{i}}}\frac{1}{1+\phi_{k\bar{k}}}|\phi_{k\bar{k}i}|^2\\
&\leq \sigma_{\epsilon,\beta} \sum_{i,j,k} \frac{1}{1+\phi_{i\bar{i}}}\frac{1}{1+\phi_{j\bar{j}}}|\phi_{i\bar{j}k}|^2
\end{align*}

By the formula $\Delta_{\omega_{\phi_\epsilon}} \text{log } \sigma_{\epsilon,\beta}=\frac{\Delta_{\omega_{\phi_\epsilon}}\sigma_{\epsilon,\beta}}{\sigma_{\epsilon,\beta}}-\frac{|\nabla \sigma_{\epsilon,\beta}|^2}{\sigma_{\epsilon,\beta}^2}$, together with the lower bound on $\sigma_{\epsilon,\beta}=\text{tr}_{\omega_{\beta,\epsilon}}\omega_{\phi_\epsilon}\geq (n-1)C$(here, we assume $n\geq 2$; for $n=1$ case, the $C^2$ bound is straightforward), we could deduce the inequality

\begin{align*}
\Delta_{\omega_{\phi_\epsilon}}\text{log }\sigma_{\epsilon,\beta}
&\geq \frac{-C\sigma_{\epsilon,\beta}-C}{\sigma_{\epsilon,\beta}}+\frac{1}{\sigma_{\epsilon,\beta}}\sum_{i,j,k}\frac{1}{1+\phi_{i\bar{i}}}\frac{1}{1+\phi_{j\bar{j}}}|\phi_{i\bar{j}k}|^2-\frac{|\nabla \sigma_{\epsilon,\beta}|^2}{\sigma_{\epsilon,\beta}^2}-C\frac{(|z|^2+\epsilon)^{-\beta}}{\sigma_{\epsilon,\beta}}\\
 &\geq -C-\frac{C}{\sigma_{\epsilon,\beta}}-\frac{(|z|^2+\epsilon)^{-\beta}}{\sigma_{\epsilon,\beta}}\\
 &\geq -C-C(|z|^2+\epsilon)^{-\beta}
 \end{align*}

\end{proof}

%%%%
%%%The RHS of the above inequality is only uniformly integrable in the case of the angle less than 1/2, which we could deal with by apply moser iteration directly.
%%

%       We need to add an auxiliary term $\Delta_{\omega_{\phi_\epsilon}} M (|z|^2+\epsilon)^p$ to compensate the blowing up term of the right hand side $-(|z|^2+\epsilon)^{-\beta}$, %        however, since $\sqrt{-1}\partial\bar\partial M(|z|^2+\epsilon)^p=\sqrt{-1} Mp(|z|^2+\epsilon)^{p-2}(p|z|^2+\epsilon)\mathrm{d}z\wedge\mathrm{d}\bar{z}\geq Mp^2(|z|^2+
%         epsilon)^{p-1}$ for $p\in (0,1)$. Using the rough upper bound $\omega_{\phi_\epsilon}\leq C\frac{\omega_0}{(|z|^2+\epsilon)^{1-\beta}}$, we have $\Delta_{\omega_{\phi_
%          \epsilon}} %M(|z|^2+\epsilon)^p\geq C^{-1} Mp^2(|z|^2+\epsilon)^{p-\beta}$ which has exponent $p-\beta$ greater than $-\beta$, does not give us the compensation........

%However, we could perturb the angle a little bit and get some extra room to cancel out this blow up term.

The rough bound $C^{-1}\omega_0\leq \omega_{\phi_\epsilon}\leq C\frac{\omega_0}{(|z|^2+\epsilon)^{1-\beta}}$ is not enough for our application.  We may first use the rougher lower bound $\omega_{\phi_\epsilon}\leq C\frac{\omega_0}{(|z|^2+\epsilon)^{1-\beta'}}$ by decreasing the angle a little bit to get a better control on the tangential direction and then conclude the \emph{correct} upper bound on the direction perpendicular to the divisor.\\

%%
%%%
%%%% The new method of using auxiliary function to overcome the blow up term
%%%
%%

\subsection{Decreasing the angle a little}

Let's rewrite the equation 

\begin{equation}
\omega_{\phi_\epsilon}^n=e^{F_{\epsilon,\beta'}}\omega_{\beta',\epsilon}^n
\end{equation}

where $\beta'$ is a number a little bit smaller than $\beta$, and $F_{\epsilon,\beta'}=F_{\epsilon,\beta}+(\beta-\beta')\text{log }(|z|^2+\epsilon)$.

Then we have a similar inequality as in the above section:

\begin{prop}
$$\Delta_{\omega_{\phi_\epsilon}} \emph{log } \sigma_{\epsilon,\beta'}\geq -C(|z|^2+\epsilon)^{-\beta'}-C$$
\end{prop}

\begin{proof}{(Proof of Proposition 4)   }
 Since $\sqrt{-1}\partial\bar{\partial}F_{\epsilon,\beta'}=\sqrt{-1}\partial\bar{\partial}F_{\epsilon,\beta}+\sqrt{-1}\partial\bar{\partial}\text{log }(|z|^2+\epsilon)
\geq \sqrt{-1}\partial\bar{\partial}F_{\epsilon,\beta}\geq -\sqrt{-1}C\{\frac{1}{|z|^2+\epsilon}\mathrm{d}z\wedge\mathrm{d}\bar{z}+\sum_i\mathrm{d}z_i\wedge\mathrm{d}\bar{z}_i\}+\text{mixed terms}$, we have $\Delta_{\omega_{\beta',\epsilon}}F_{\epsilon,\beta'}=\text{tr }_{\omega_{\beta',\epsilon}} \sqrt{-1}\partial\bar\partial F_{\epsilon,\beta'}\geq -C-C(|z|^2+\epsilon)^{1-\beta'}(|z|^2+\epsilon)^{-1}=-C-C(|z|^2+\epsilon)^{-\beta'}$.
\end{proof}

Now modify the differential inequality by adding an auxiliary function $C(\beta-\beta')^{-2}(|z|^2+\epsilon)^{\beta-\beta'}$, and use the \emph{rough bound} $\omega_{\phi_\epsilon}\leq C\frac{\omega_0}{(|z|^2+\epsilon)^{1-\beta}}$, we get the inequality:

\begin{align*}
\Delta_{\omega_{\phi_\epsilon}} &\{\text{log }\sigma_{\epsilon,\beta'} +C(\beta-\beta')^{-2}(|z|^2+\epsilon)^{\beta-\beta'}\}\\
&\geq C(|z|^2+\epsilon)^{-\beta'}-C(|z|^2+\epsilon)^{-\beta'}-C\\
&\geq-C 
\end{align*}

Since $\Delta_{\omega_{\phi_\epsilon}}\{C\beta^{-2}(|z|^2+\epsilon)^\beta\}\geq C$, the final inequality will be the following proposition:

\begin{prop}
$$\Delta_{\omega_{\phi_\epsilon}} \big\{\text{log }\sigma_{\epsilon,\beta'}+C(\beta-\beta')^{-2}(|z|^2+\epsilon)^{\beta-\beta'}+C\beta^{-2}(|z|^2+\epsilon)^\beta\big\} \geq 0$$
\end{prop} 

\vspace{0.2cm}
%%%
%%%Apply Moser's Iteration to the Trace function
%%%

\subsection{Moser's Iteration}

%%%
%%%Inequality for Moser's Iteration

We are going to apply the local Moser Iteration to the function $f=\text{log }\sigma_{\epsilon,\beta'}+C(\beta-\beta')^{-2}(|z|^2+\epsilon)^{\beta-\beta'}+C\beta^{-2}(|z|^2+\epsilon)^\beta$, notice that since $\sigma_{\epsilon,\beta'}=\text{tr }_{\omega_{\beta',\epsilon}} \omega_{\phi_\epsilon}\geq \text{tr }_{\omega_{\beta',\epsilon}} C^{-1}\omega_0\geq (n-1)C^{-1}$, we could add a constant to $f$ to make it positive.\\

Now a local inequality is settled up:
\begin{equation}
\Delta_{\omega_{\phi_\epsilon}} f \geq 0
\end{equation}

We are going to do Moser Iteration on the pair of $\omega_{\phi_\epsilon}$-metric balls $V_\alpha$ and $U_\alpha$, since there are only finitely many, there is no loss of generality to assume $V_\alpha$ is radius 1 and $U_\alpha$ is radius 2. The $L^\infty$ norm of $f$ on metric $1$-ball will be bounded by the $L^2$ bound of $f$ on metric $2$-ball,  and it is carried out out below, which ignores $\epsilon$ for simplicity:

%%%%%%
%%%%%
%%%%         Moser Iteration to get a priori C^0 bound
%%%
%%
%
\begin{align*}
\frac{4p}{(p+1)^2} \int_X \eta^2 |\nabla f^\frac{p+1}{2}|^2\omega_\phi^n
&=p\int_X \eta^2 f^{p-1}|\nabla f|^2 \omega_{\phi}^n\\
&= \int_X \{ \nabla(\eta^2 f^p \nabla f)-2\eta f^p \nabla \eta \cdot \nabla f-\eta^2 f^p\Delta_{\omega_\phi} f \}\omega_\phi^n\\
&\leq \int_X -2\eta f^p \nabla \eta \cdot \nabla f \omega_\phi^n\\
&=\int_X -\frac{4}{p+1}  f^\frac{p+1}{2}\nabla \eta \cdot \eta \nabla f^\frac{p+1}{2}\quad  \omega_\phi^n\\
&\leq \int_X \frac{2}{p+1} \{\delta \eta^2|\nabla f^\frac{p+1}{2}|^2 +\delta^{-1}f^{p+1}|\nabla \eta|^2\} \omega_\phi^n
\end{align*}

Then by taking $\delta=\frac{p}{p+1}$ , we would end with the control 
\[
\int_X \eta^2 |\nabla f^\frac{p+1}{2}|^2 \omega_\phi^n \leq (\frac{p+1}{p})^2\int_X f^{p+1}|\nabla \eta|^2\omega_\phi^n\leq 4\int_X f^{p+1}|\nabla \eta|^2\omega_\phi^n
\]

Since our manifold $(M,\omega_{\phi_\epsilon})$ has a uniform positive lower bound on the Ricci curvature and uniform constant volume, we have a uniform Sobolev constant, which means:

%%%%
%%%
%%% Sobolev  Inequality here!
%%%
%%%

\begin{align*}
\{\int_X |\eta f^\frac{p+1}{2}|^\frac{2n}{n-1} \omega_\phi^n\}^\frac{n-1}{n} 
&\leq C_S \int_X \{ \eta^2 f^{p+1}+|f^\frac{p+1}{2} \nabla \eta+\eta \nabla f^\frac{p+1}{2}|^2 \}\omega_\phi^n\\
&\leq C_S \int_X \{ \eta^2 f^{p+1}+2(\eta^2 |\nabla f^\frac{p+1}{2}|^2+f^{p+1}|\nabla \eta|^2)\} \omega_\phi^n\\
&\leq 10 C_S \int_X (\eta^2+|\nabla \eta|^2) f^{p+1}\omega_\phi^n
\end{align*}

Now by taking a suitable cut-off function on the real line $\eta$ and compositing with the distance function on the manifold we get a cut-off function which satisfies $\eta\equiv 0$ outside the metric ball $B_\phi (R)$ and $\eta\equiv 1$ inside the metric ball $B_\phi(S)$ and $|\nabla \eta|\leq \frac{C}{R-S}$, plug in this cut-off function to the above inequality we get:

\[
||f||_{L^{(p+1)\frac{n}{n-1}}(B_\phi(r_1))}\leq (10C_S)^\frac{1}{p+1}\{1+\frac{C}{(R-S)^2}\}^\frac{1}{p+1}||f||_{L^{p+1}(B_\phi(r_2))}
\]

The Moser Iteration technique uses $\gamma_m$ to replace $p+1$ in the above inequality and use the pair of radius $r_m=1+2^{-m}$ and $r_{m+1}$ to replace $R$ and $S$ on the $m$-th step, and eventually will end up with 

\begin{equation}
||f||_{L^\infty(B_\phi(1))} \leq C ||f||_{L^2(B_\phi(2))}
\end{equation}\\

%%%
%%%Metric Ball verse Coordinate Ball
%%%

Since $\sigma_{\epsilon,\beta'}\leq C(1+(|z|^2+\epsilon)^{\beta-1})\leq C(|z|^2+\epsilon)^{\beta-1}$, the $L^2$ bound of $f$ on the unit $\omega_{\phi_\epsilon}$-metric ball is bounded by the following Lemma:
 
 \begin{lem}
 $$\int_{B_{\phi_\epsilon}(p,2)} |\text{log }\sigma_{\epsilon,\beta'}|^2 \omega_{\phi_\epsilon}^n
 \leq C\int_{|z|\leq 2\sqrt{C}} |\text{log} (|z|^2+\epsilon)|^2 (|z|^2+\epsilon)^{\beta-1} \sqrt{-1}\mathrm{d}z\wedge\mathrm{d}\bar{z}\leq C$$
 
 $$\int_{B_{\phi_\epsilon}(p,2)} |\text{log }\sigma_{\epsilon,\beta'}|(|z|^2+\epsilon)^{\beta-\beta'} \omega_{\phi_\epsilon}^n 
 \leq C\int_{|z|\leq 2\sqrt{C}} |\text{log} (|z|^2+\epsilon)| (|z|^2+\epsilon)^{2\beta-\beta'-1} \sqrt{-1}\mathrm{d}z\wedge\mathrm{d}\bar{z}\leq C$$
 
 $$\int_{B_{\phi_\epsilon}(p,2)} |\text{log }\sigma_{\epsilon,\beta'}|(|z|^2+\epsilon)^{\beta} \omega_{\phi_\epsilon}^n 
 \leq C\int_{|z|\leq 2\sqrt{C}} |\text{log} (|z|^2+\epsilon)| (|z|^2+\epsilon)^{2\beta-1} \sqrt{-1}\mathrm{d}z\wedge\mathrm{d}\bar{z}\leq C$$

 \end{lem}

Since $D$ is covered by finitely many coordinate balls $\mathcal{V}_\alpha$,  and on each of those balls $tr_{\omega_{\beta',\epsilon}} \omega_{\phi_\epsilon} \leq C$.
Comparing with the rough bound $C^{-1}\omega_0\leq \omega_{\phi_\epsilon} \leq C \frac{\omega_0}{(|z|^2+\epsilon)^{1-\beta}}$ achieved in Proposition 2, this yields us a finer control about the metric $\omega_{\phi_\epsilon}$ on the tangential direction. $\text{tr}_{\omega_{\beta',\epsilon}}\omega_{\phi_\epsilon}\leq C$ implies that $\omega_{\phi_\epsilon}\leq C\omega_{\beta',\epsilon}$, thus we have uniform bound on the divisor direction $\lambda_2,\cdots,\lambda_n\leq C$. Combined with the previous rough bound $\lambda_2,\cdots,\lambda_n\geq C^{-1}$  and the asymptotic behavior of the volume form $\lambda_1\lambda_2\cdots\lambda_n\sim \frac{1}{(|z|^2+\epsilon)^{1-\beta}}$, we conclude that $\lambda_1\sim \frac{1}{(|z|^2+\epsilon)^{1-\beta}}$. Finally we get the desired uniform control about the metric $\omega_{\phi_\epsilon}$.

\begin{thm}
For a uniform constant $C$, we have control $C^{-1}\omega_{\beta,\epsilon}^\alpha\leq \omega_{\phi_\epsilon} \leq C\omega_{\beta,\epsilon}^\alpha$ on $\mathcal{V}_\alpha$.
\end{thm}

The standard Evans-Krylov theory outside $D$ gives us high order estimate of $\omega_{\phi_\epsilon}$ which will yield a smooth limit K\"{a}hler-Einstein metric on $X\backslash D$. By letting $\epsilon\to 0$ in the above theorem, we get the theorem:

\begin{thm}
On a Fano manifold $X$,let $D$ is a smooth anti-canonical divisor, $\beta \in (0,1)$, if the twisted K-energy $E_{(1-\beta)D}$is proper, then there exists a weak conical K\"{a}hler-Einstein metric in $c_1(X)$ with angle $2\pi \beta$ along $D$.
\end{thm}

%%%%
%%%     General situation involves the same technique, where we get similar result for Kahler Einstein metric with negative, zero and positive curvature
%%
%

\section{General Current Equation}

On a quite general situation of compact K\"{a}hler manifolds, we could set up similar problem of prescribing K\"{a}hler-Einstein metrics on $X$ with conical singularities along  a smooth hypersurface $D$. Pick a smooth background K\"{a}hler metric $\omega_0$ in a fixed cohomology class $\mathfrak{W}$, and suppose $\mathfrak{A}=c_1(L_D)$  satisfies the cohomological condition: 

$$c_1(X)=\mu \mathfrak{W}+(1-\beta)\mathfrak{A}$$

Let

$$Ric(\omega_0)=\mu \omega_0+(1-\beta)\alpha_0$$

where $\alpha_0$ is a smooth $(1,1)$ form in $\mathfrak{A}$ and it is the curvature form of some Hermitian metric $h$ on $L_D$. Let $S$ be holomorphic section of $L_D$ which defines $D$, and in $\mathfrak{A}$ we define a family of smooth $(1,1)$ form $\chi_\epsilon=\alpha_0+\sqrt{-1}\partial\bar\partial \text{log }(|S|_h^2+\epsilon)$, notice that $\chi_\epsilon$ approaches the \emph{delta distribution} $\chi$ along $D$ in the current sense as $\epsilon \to 0$. Moreover, as a $(1,1)$ form, $\chi_\epsilon$ has lower bound:

\begin{align*}
\chi_\epsilon=\alpha_0+\sqrt{-1}\partial \bar\partial \text{log }(|S|_h^2+\epsilon)
&=\alpha_0+\sqrt{-1}\partial\frac{|S|_h^2\bar{\partial}\text{log }|S|_h^2}{|S|_h^2+\epsilon}\\
&=\alpha_0+\frac{|S|_h^2}{|S|_h^2+\epsilon}\sqrt{-1}\partial\bar\partial \text{log }|S|_h^2+\sqrt{-1}\frac{\epsilon}{|S|_h^2(|S|_h^2+\epsilon)^2}\partial |S|_h^2\wedge \bar\partial |S|_h^2\\
&\geq \alpha_0+\frac{|S|_h^2}{|S|_h^2+\epsilon}(-\alpha_0)\\
&=\frac{\epsilon}{|S|_h^2+\epsilon}\alpha_0
\end{align*}

In order to solve the singular equation

\begin{equation}
Ric(\omega)=\mu \omega+(1-\beta)\chi,
\end{equation}

we first solve the smoothing equation instead,

\begin{equation}
Ric(\omega_{\phi_\epsilon})=\mu \omega_{\phi_\epsilon}+(1-\beta)\chi_\epsilon
\end{equation}\\

and then consider the limit metric as $\epsilon\to 0$.

For any $\epsilon>0$, we could set up the \emph{Continuity Path} $\star_\epsilon$:

\begin{equation*}
Ric(\omega_{\phi_\epsilon(t)})=(1-t)\{\mu \omega_0+(1-\beta)\alpha_0\}+t\{\mu {\omega_{\phi_\epsilon(t)}}+(1-\beta)\chi_\epsilon\}
\end{equation*}

%%%
%%%   The argument why this general situation is the same as before if we have C^0 bound already
%%%
%%%

The solvability of $\star_\epsilon$ is different for the cases $\mu<0, \mu=0$ and $\mu>0$ and the below section distincts them, and will be treated differently.\\

%%%
%%%   How to get C^0 bound
%%

%%%%%%%   CASE 1

Case 1: $\mu<0$. By standard argument of \emph{Continuity Method}, $\star_\epsilon$ could be solved up to $t=1$. 

The openness part follows from the injectivity of the operator $\Delta_{\omega_{\phi_\epsilon(t)}}+t\mu$ since $\mu<0$.

The $C^0$ bound is by combining \emph{Maximum Principle} to the equation $$\omega_{\phi_\epsilon(t)}^n=e^{-t\mu \phi_\epsilon(t)+C_{\epsilon,t}}\frac{\omega_0^n}{(|S|_h^2+\epsilon)^{(1-\beta)t}}$$ 
and Kolodziej's $L^p$ estimate.
Concretely speaking, first solve the equation for $t\in [0,\delta_0]$ by openness and then by considering the maximum point of $\phi_\epsilon(t)$ we know that

 $$\sup e^{-t\mu \phi_\epsilon(t)+C_{\epsilon,t}}\leq \sup (|S|_h^2+\epsilon)^{(1-\beta)t}\leq M.$$
 
The right hand of the equation is uniform in $L^p$ therefore we could apply $L^p$ estimate to achive $C^0$ estimate for the potential $\phi_\epsilon(t)$. Eventually, we could solve the equation up to $t=1$ with uniform $C^0$ bound on $\phi_\epsilon$. And the higher order bound (depending on $\epsilon$) is obtained by Aubin.  \\

%%%%%%%%  CASE 2

Case 2: $\mu=0$. We use the smooth Calabi's conjecture instead of continuity method.

\begin{thm}
Suppose $D$ is a smooth hypersurface in the K\"{a}hler manifold $X$, and $c_1(X)=(1-\beta)c_1(L_D)$, then we have a \emph{weak} Ricci flat K\"{a}hler metric with conical singularities of angle $2\pi \beta$  in any K\"{a}hler class.
\end{thm}

\begin{proof}

The equation $\omega_{\phi_\epsilon}^n=e^{h_{\omega_0}+C_\epsilon}\frac{\omega_0^n}{(|S|_h^2+\epsilon)^{1-\beta}}$ , where $C_\epsilon$ is a normalization constant such that $\int_X \frac{e^{h_{\omega_0+C_\epsilon}}}{(|S|_h^2+\epsilon)^{1-\beta}}\omega_0^n=\int_X \omega_0^n=1$,  could be solved by Yau's solution on the Calabi's Conjecture. The right hand side has uniform $L^p$ bound for $p<\frac{1}{1-\beta}$.  Using Kolodziej's local estimate, we could get a uniform $C^0$ bound on the whole manifold.

From $C^0$ estimate to $C^2$ estimate the above technique could be used.  $Ric(\omega_{\phi_\epsilon})=(1-\beta)\alpha_0+(1-\beta)\sqrt{-1}\partial\bar\partial \text{log }(|S|_h^2+\epsilon)\geq C(1-\beta)\omega_0$, then the Chern-Lu inequality (the version that $Ric(\omega_\phi)\geq -C_1\omega_\phi-C_2 \omega_0$, see Lemma 1) gives the \emph{rough estimate} first.  From this, the Ricci curvature is bounded form below since the rough bound $\omega_{\phi_\epsilon}\geq C^{-1}\omega_0$ tells that $\text{tr}_{\omega_{\phi_\epsilon}}\omega_0\leq C$. At the same time, since the rough bound also gives the uniform diameter upper bound, we have uniform Sobolev constant.  Then the technique used above could be applied to prove uniform $C^2$ estimate.
\end{proof}

\begin{rmk}
Ricci flat K\"{a}hler metric with conic singularities was first considered by Brendle under restrictions of the cone angle $\beta \in (0,\frac{1}{2})$, see \cite{Brendle}.  The above theorem removes the assumption on the cone angle while getting a Ricci flat K\"{a}hler metrics with conical singularities with weaker regularities than the result of Brendle \cite{Brendle}, see Section 3 for discussions.
\end{rmk}

%%%%%%%%%  CASE 3

Case 3: $\mu>0$ is obstructed as before, the solvability need the \emph{Properness} of the twisted K-Energy $E_{(1-\beta)\chi}$.  In this situation, the smooth continuity method could also be solved up to $t=1$ with uniform $C^0$ estimate by the same reasoning explained in section 2 above.

\begin{rmk}
The $C^0$ bound for the case $\mu<0$ and $\mu=0$ do not depend on positivity of $D$, which implies the \emph{rough estimate} from which we could get the lower bound on the Ricci curvature therefore the uniform upper bound on the Sobolev constant. However, for $\mu>0$, from bound on $I-J$ to $C^0$ bound, we need the uniform Sobolev constant. That's why in the Main Theorem $c_1(L_D)\geq 0$ is assumed. 
\end{rmk}

As soon as we have uniform $C^0$ bound on $\phi_\epsilon$ for equation $(8)$, we are ready to prove the uniform $C^2$ estimate. Since $\chi_\epsilon\geq \frac{\epsilon}{|S|_h^2+\epsilon}\alpha_0\geq -C\omega_0$, $Ric(\omega_{\phi_\epsilon})=\mu \omega_{\phi_\epsilon}+(1-\beta)\chi_\epsilon \geq \mu\omega_{\phi_\epsilon}-C(1-\beta)\omega_0$, the Chern-Lu inequality (See Lemma 1 above) gives the \emph{rough estimate}. Afterwards, we could prove the differential inequality , since the essential part is the calculation of the lower bound of $\Delta_{\omega_{\beta,\epsilon}}\text{log }\frac{|z|^2+\epsilon}{|z|^2+\epsilon e^h}$. And because $\sqrt{-1}\partial\bar{\partial}h=\alpha_0$ is a fixed smooth $(1,1)$ form, $h_{z\bar{z}}$ and $h_z$ is uniformly bounded, a similar argument tells us this quantity is bounded below by $-\frac{C}{|z|^2+\epsilon}\mathrm{d}z\wedge\mathrm{d}\bar{z}-C\sum_i\mathrm{d}z_i\wedge\mathrm{d}\bar{z}_i$. 

Another piece of information is the uniform Sobolev constants. The Ricci curvature is bounded below by $\mu-C(1-\beta)$ since $\text{tr}_{\omega_{\phi_\epsilon}}\omega_0\leq C$ by the rough bound $\omega_0 \leq C\omega_{\phi_\epsilon}$.  The lower bound on the Ricci curvature together with the diameter upper bound (Lemma 3) and the constant volume gives us the uniform Sobolev constant, thus the \emph{a priori} estimate on second derivative by $\emph{local Moser Iteration}$ goes through.\\

%%%%
%%%     Ricci flat Kahler metrics in the first Chern Class along hypersurface inside a Fano manifolds. 
%%%
%%%
%%%

%%%  Higher Order Regualarity?
%%%

\section{Regularity}

By studying Calabi's third order estimate, Brendle showed the regularity of Ricci-flat K\"{a}hler metrics he obtained under the assumption $\beta \in(0,\frac{1}{2})$. In \cite{Yanir}, the group of authors studied the \emph{a priori $C^{2,\alpha}$ estimate} along the continuity path of conical singular K\"{a}hler metrics with edge and wedge calculus. And in \cite{Calamai}, Calamai and Zheng adapted Evans-Krylov theory to the conical setting to prove the \emph{a priori} $C^{2,\alpha}$ estimate for the continuity path of K\"{a}hler-Einstein metrics with conical singularities under a technical assumption that $\beta<\frac{2}{3}$. It is still interesting to know how to get the general regularity directly from the \emph{quasi-isometry} condition, i.e. whether we could get some uniform H\"{o}lder estimate of $\omega_{\phi_\epsilon}$ with respect to the local model metric $\omega_{\beta,\epsilon}$ or third order estimate. \\

When $\beta=1$, the weak K\"{a}hler-Einstein metric is actually smooth as shown in the appendix of \cite{Chen3}, the natural question is that whether it is \emph{smooth} with respect to the background conical K\"{a}hler metric for the general conical angle $2\pi \beta$, i.e. whether it is $C^\alpha$ in the singular coordinate around $D$? \\

\vspace{2cm}

\indent{Chengjian Yao}\\
\indent{Department of Mathematics, Stony Brook University. Stony Brook, 11790, NY, USA.}
\indent{E-mail: yao@math.sunysb.edu}

\end{document}